\documentclass[sn-mathphys,Numbered]{sn-jnl}

\usepackage{graphicx}%
\usepackage{multirow}%
\usepackage{amsmath,amssymb,amsfonts}%
\usepackage{amsthm}%
\usepackage{mathrsfs}%
\usepackage[title]{appendix}%
\usepackage{xcolor}%
\usepackage{textcomp}%
\usepackage{manyfoot}%
\usepackage{booktabs}%
\usepackage{algorithm}%
\usepackage{algorithmicx}%
\usepackage{algpseudocode}%
\usepackage{listings}%
\usepackage{bm}%
\usepackage{enumitem}
\usepackage{caption}
\def\div{\text{div$\,$}}
\def\ep{\bm \epsilon}
\def\sig{\bm \sigma}
\def\bmU{\boldsymbol U}
\def\bmH{\boldsymbol H}
\def\bmf{\boldsymbol f}

\def\bmu{\boldsymbol u}
\def\bmv{\boldsymbol v}
\def\bmw{\boldsymbol w}
\def\bmn{\boldsymbol n}
\def\bmx{\boldsymbol x}
\def\bme{\boldsymbol e}
\def\euo{\bme_{\bmu}^i}
\def\epo{e_p^i}
\def\eun{\bme_{\bmu}^{i+1}}
\def\epn{e_p^{i+1}}

\theoremstyle{thmstyleone}%
\newtheorem{theorem}{Theorem}%  meant for continuous numbers
\newtheorem{lemma}[theorem]{Lemma}% 

\theoremstyle{thmstyletwo}%
\newtheorem{remark}{Remark}%

\theoremstyle{thmstylethree}%
\newtheorem{assumption}{Assumption}%

\begin{document}

\title[Fixed-stress split method for nonlinear poroelasticity]{A fixed-stress type splitting method for nonlinear poroelasticity}


\author[1,1]{\fnm{Johannes} \sur{Kraus}}\email{johannes.kraus@uni-due.de}
\equalcont{These authors contributed equally to this work.}

\author[2,2]{\fnm{Kundan} \sur{Kumar}}\email{Kundan.Kumar@uib.no}
\equalcont{These authors contributed equally to this work.}

\author[3,3]{\fnm{Maria} \sur{Lymbery}}\email{mariadimitrova.lymbery@uk-essen.de}
\equalcont{These authors contributed equally to this work.}

\author*[4,2]{\fnm{Florin A.} \sur{Radu}}\email{Florin.Radu@uib.no}
\equalcont{These authors contributed equally to this work.}

\affil[1]{\orgdiv{Faculty of Mathematics}, \orgname{University of Duisburg-Essen}, \orgaddress{\street{Thea-Leymann-Stra\ss e 9}, \city{Essen}, \postcode{45127}, \country{Germany}}}

\affil*[2]{\orgdiv{Center for Modeling of Coupled Subsurface Dynamics, Department of Mathematics}, \orgname{University of Bergen}, \orgaddress{\street{All\'{e}gaten 41}, \city{Bergen}, \postcode{5020}, \country{Norway}}}

\affil[3]{\orgdiv{Institute for Artificial Intelligence in Medicine}, \orgname{University Hospital Essen}, \orgaddress{\street{Girardetstra\ss e 2}, \city{Essen}, \postcode{45131}, \country{Germany}}}


\abstract{
%The abstract serves both as a general introduction to the topic and as a brief,
%non-technical summary of the main results and their implications.
%Authors are advised to check the author instructions for the journal they are submitting
%to for word limits and if structural elements like subheadings, citations, or equations are permitted.
In this paper we consider a nonlinear poroelasticity model that describes the quasi-static mechanical
behaviour of a fluid-saturated porous medium whose permeability depends on the divergence of the
displacement. Such nonlinear models are typically used to study biological structures like tissues,
organs, cartilage and bones, which are known for a nonlinear dependence of their permeability/hydraulic
conductivity on solid dilation.

We formulate (extend to the present situation) one of the most popular splitting schemes, namely
the fixed-stress split method for the iterative solution of the coupled problem. The method is proven
to converge linearly for sufficiently small time steps under standard assumptions. The error contraction
factor then is strictly less than one, independent of the Lam\'{e} parameters, Biot and storage coefficients
if the hydraulic conductivity is a strictly positive, bounded and Lipschitz-continuous function.
}

\keywords{nonlinear Biot model, soft-material poromechanics, fixed-stress splitting, convergence analysis, 
parameter robustness}



\maketitle

\section{Introduction}\label{sec1}

% Here we discuss relevant related works, say what is new in this paper, and put the present work into a proper context.
% Finally we give a short outline. 
% Introduction consists of 5 paragraphs.
% \begin{itemize}
% \item  Why should we be interested in solving
% permeability dependent on div u.
% \item There are three main approaches...explicit, implicit, and
% iterative. A coupled model requires iterative methods. There are 4
% types of iterative methods. Fixed stress is the most popular one.
% \item What have we done in this paper. Why is this important.
% \item What are the other related works that we have done. How does
% this work fit in there?
% \item The organization of the paper.
% \end{itemize}
%  Coupled flow and mechanics is an important problem. 
Coupling of flow and mechanical deformations in a porous medium,
referred as poroelasticity,  is relevant 
%for 
{\color{teal}to}
several applications.
%This includes
These include subsurface deformations resulting from hydrocarbon
recovery \cite{detournay1993fundamentals, DetournayCheng1988} or anthropogenically induced seismic events due to geothermal
stimulations \cite{jha2014coupled}. In biological context, there are several examples where
such a coupling is needed. For example, this is used to model the
interaction between the fluid flow within the bone's microstructure
due to interstitial fluid and the mechanical deformation of the bone \cite{cowin1999bone}.
Similarly, the tumor growth and angiogenesis involves the interaction
of solid tumor tissue, interstitial fluid flow, and angiogenesis \cite{roose2003solid}.
Cartilage mechanics \cite{carter2003modelling}, biomechanics of soft tissues \cite{konofagou2001poroelastography},  brain tissues \cite{kyriacou2002brain, bohr2022glymphatic} are
other examples where the coupling of fluid flow and the deformation of
solid matrix needs to be taken into account. 

% Biot model: setting coupled flow and mechanics framework.
\par This paper concerns the numerical solution of the coupled flow and mechanical deformation in the context of poroelasticity. The Biot equations are the most common approach of modelling poroelasticity. In a quasi-static model  they comprise  a set of partial differential equations describing the flow and displacement using the physical laws of
 mass and momentum conservation,  see Section \ref{subsec2.1} for the model equations. Together with linear constitutive models, the
mechanical deformations are captured by an elliptic (linear elasticity) equation for the displacement and the fluid pressure is
described by a parabolic equation for the evolution of the pressure.
The coupling terms quantify the deformation and stress fields due to changes in
the fluid pressure in addition to taking into account the changes in
porosity due to mechanical deformations. Starting with the pioneering works of  Terzaghi \cite{terzaghi1996soil} and Biot \cite{Biot1941}, it was used to model the  consolidation of soils \cite{Biot1941, Biot1955}.  A comprehensive discussion of the theory of poromechanics can be found in \cite{Coussy1995}. The standard reference for the well-posedness theory for the Biot model is \cite{Showalter2000} and for more recent extensions, we refer to  \cite{ShowalterStefanelli2004}.  The quasi-static Biot model  can be regarded as the singular limit of the fully
dynamic Biot-Allard problem  \cite{MikelicWheeler2012} after neglecting the acceleration of solid in the mechanics part.

\par In this work, we consider the case when the hydraulic conductivity $K$  changes due
to the evolving stress conditions, that is,  upon the dilatation (local volume change), $\div \bmu$, i.e., $K = K(\div \bmu)$.   The ease with which a fluid flows through the solid skeleton of porous media is characterized by the hydraulic conductivity tensor. It is a function of porosity and tortuosity (a measure of geometric complexity of the porous medium, see Karmen Kozeny relationships). Changes in porosity due to mechanical deformations are captured using the term $\div \bmu$. Accordingly, the hydraulic conductivity is a function of $\div \bmu$.  Such a stress dependent permeability model was introduced first in \cite{hiltunen1995mathematical} to describe poroelasticity effects in the paper production process. Same type of model is used to account for stress sensitive reservoirs in petroleum related applications \cite{raghavan2002productivity, chin2000fully}. Mathematically, these extended Biot's models have been studied in \cite{barbeiro2010priori, cao2013analysis, cao2014steady, cao2015quasilinear},  where the authors have proved existence and uniqueness of solutions and developed the numerical approximations and their analysis. A comprehensive study of a nonlinear Biot model including the stress dependent permeability case and variations of boundary conditions can be found in \cite{bociu2016analysis}. For a recent work developing further the theory of nonlinear single phase poroelasticity, we refer to \cite{van2023mathematical}. Numerically, extensions of single phase nonlinear Biot models have been treated in \cite{borregales2018robust, borregales2019partially, borregales2021iterative}.

% How are we supposed to solve this model? Why consider iterative methods. 
 To solve the Biot model numerically,   there are two main categories of schemes:
fully implicit and fully explicit. The fully implicit coupling
approach considers fluid variables such as pressure and kinematic
variable displacement as unknowns at each time step and linearizes the
fully coupled system to obtain the Jacobian. This approach allows one
to take larger time steps and provides more stability. However, there
are some important disadvantages. First, the coupled linear system is
difficult to solve, especially in multiphase flows where the operator
of the geomechanics noticeably differs from that of the flow problem.
Second, a fully coupled system may lose some flexibility, for
instance, it does not permit the use of large time steps for the
geomechanical response compared to the flow time step (this may be
exploited to develop multirate schemes \cite{Kumaretal2015, Almanietal2016}). Finally, the implementation
of a fully implicit approach is more difficult compared to individual
equations.  In contrast to the implicit approach, an explicit coupling
approach allows one to take the output from the mechanics (e.g.,
displacement) and use it as an input for the flow problem. This is 
simpler to implement and also allows more flexibility, for example in
choosing time steps for individual equations separately. However,
performing a naive explicit decoupling leads to unstable schemes or at
best conditionally stable schemes. 

\par An elegant way to combine the advantages of the above two broad approaches is to consider an iterative scheme. In this, at each step in time, the flow problem is solved followed by the mechanics problem using the pressure from the first step. 
The procedure is repeated until the desired convergence is reached. Since we treat the equations separately, both the ease
of implementation and flexibility are retained.  However, to ensure stability and robustness, the design of an 
iterative scheme demands careful considerations.   For the fully-coupled problems, iterative methods can be used to construct  efficient preconditioning techniques for the arising algebraic systems, the investigation of 
which is a matter of active ongoing scientific research, (see, e.g.,  \cite{Allen2009,Whiteetall2016,CastellettoWhiteFerronato2016,GasparRodrigo2017}) including the recently developed parameter-robust methods in \cite{hong2019conservative, LeeMardalWinther2017, HongKraus2017, hong2020parameter,Hong_parameter-robust_2020}). 

\par Two main iterative coupling schemes to solve the flow problem coupled with geomechanics in an iteratively sequential manner are the fixed-stress split and the undrained split schemes \cite{SettariMourits1998,hong2020parameter,Allen2009} (see \cite{both2019gradient} for a variational derivation for these iterative schemes).  In the fixed-stress split iterative scheme \cite{KimTchelepiJuanes2011FixedStress, MikelicWheeler2012}, a constant volumetric mean total stress is assumed during the flow solve, whereas in the undrained split scheme \cite{KimTchelepiJuanes2011DrainedSplit, TAKA}, a constant fluid mass is assumed during the mechanics solve. In the linear Biot case, both schemes were shown to be convergent in 
\cite{MikelicWheeler2012, both2017robust, AlmaniKumarWheeler2017} and \cite{TAKA}. Fixed stress splitting is preferred in practice because of its robustness and efficiency. In this work we focus on the extension of fixed stress split scheme to our nonlinear setting. Our main contribution is in the convergence analysis of the fixed stress splitting
scheme for the aforementioned nonlinear Biot model. The fixed stress splitting scheme introduces a stabilization of the flow equation and then iterates with the elasticity equation.  This stabilization amounts to adding a diagonal term to the pressure matrix block.  One key question is the amount of stabilization that should be added, which will be addressed in the proof of  convergence of the scheme. Moreover, our analysis uses solution variables in their natural energy norms in contrast to the approaches in \cite{MikelicWheeler2012, AlmaniKumarWheeler2017}. Though we treat here the fixed stress splitting scheme,  the idea of the proof extends to the other types as well with necessary alterations. 

The paper is organized as follows:  Section \ref{subsec2.1} introduces the nonlinear Biot equations. The fixed-stress splitting scheme is given in Section \ref{subsec2.2}. The convergence analysis is carried out in Section \ref{sec3} and the main result is stated in Section \ref{subsec3.2}. The numerical examples including convergence studies for different choices for $K(\div \bmu)$ are in Section \ref{sec4} followed by conclusions in Section \ref{sec5}.

\section{Problem formulation and splitting scheme}\label{sec2}

In this section, we present a nonlinear poroelasticity model that generalizes the classical quasi-static Biot's model of
consolidation by introducing a general (nonlinear) dependency of the hydraulic conductivity, or, equivalently, porosity
and/or permeability, on solid dilation, i.e., on the divergence of the displacement field.
We further extend a popular splitting scheme for this nonlinear model, namely, the fixed-stress split method that can
be used to solve the coupled problem iteratively by solving in each iterations step one mechanics and and one flow
equation separately.

\subsection{A nonlinear poroelasticity model}\label{subsec2.1}

We consider the following nonlinear poroelasticity problem in two-field formulation. Let $\Omega$ be a bounded Lipschitz
domain in $\mathbb{R}^{d}, d=2,3.$ Then the solid displacement $\bmu$ and fluid pressure $p$ are sought as the
solution of the differential-algebraic system
\begin{subequations}\label{eq:nlBiot}
\begin{align}
-\div \sig + \alpha \nabla p &= \bmf~~ \text{in}~~ \Omega\times (0,T),\label{nlBiot1}\\
\frac{\partial}{\partial t}\left( \alpha \div \bmu + S p\right) - \div (K(\div \bmu) \nabla p)  &=g\;\;\text{in}~~ \Omega\times (0,T),
\label{nlBiot2}
\end{align}
\end{subequations}
which couples one elliptic partial differential equation (PDE), stating the momentum balance, to one parabolic PDE,
stating the mass balance. Here $\sig=2 \mu \ep(\bmu) + \lambda \div \bmu I$ denotes the effective stress where
$\ep(\bmu) = \frac{1}{2} (\nabla \bmu + (\nabla \bmu)^T)$ is the strain tensor and $\lambda$ and $\mu$ are given Lam\'{e}
parameters. The remaining model parameters are the Biot-Willis coefficient $\alpha$, the constrained storage coefficient $S$,
and the hydraulic conductivity~$K$, which itself is defined in terms of the intrinsic permeability $\mathbb{K}$, the fluid density $\rho$
and viscosity $\eta$ (and the gravitational constant $c_g$) via the relation  $K=\mathbb{K} \rho c_g/\eta$.

Moreover, we want to make the following assumption.

\begin{assumption}\label{ass:K}
We assume that the hydraulic conductivity $K: \mathbb{R} \mapsto \mathbb{R}$ is differentiable, strictly positive, bounded
from above, and Lipschitz continuous with a Lipschitz constant $K_L$, i.e., the function~$K$ satisfies the conditions
\begin{subequations}\label{eq:ass}
\begin{align}
&K \in C^1(\mathbb{R}), \label{ass0}\\
&0 < K_0 \le K(z) \mbox{ for all } z \in \mathbb{R}, \label{ass1}\\
&K(z) \le K_1 < \infty \mbox{ for all } z \in \mathbb{R}, \label{ass2}\\
&\vert K(z_1) - K(z_2) \vert \le K_L \vert z_1 - z_2 \vert \mbox{ for all } z_1,z_2 \in \mathbb{R}. \label{ass3}
\end{align}
\end{subequations}
\end{assumption}

Then, under Assumption~\ref{ass:K} and imposing proper boundary and initial conditions, the system~\eqref{eq:nlBiot}
has a unique solution. For instance, such conditions are given by
%
%\begin{subequations}%\label{eq:nlBiot_BIC}
\begin{eqnarray*}%\nonumber
p(\bmx,t) &=& p_{D}(\bmx,t),  \qquad \bmx \in \Gamma_{p,D}, \quad t > 0,\\ %\nonumber
\frac{\partial p(\bmx,t)}{\partial \bmn} &=& q_{N}(\bmx,t),  \qquad \bmx \in \Gamma_{p,N}, \quad t > 0,\\ %\nonumber
\bmu(\bmx,t) &=& {\bmu}_D(\bmx,t),  \qquad \bmx \in \Gamma_{\bmu,D}, \quad t > 0, \\  %\nonumber
({\sig(\bmx,t)}-\alpha p  \bm I) \, {\bmn} (\bmx) &=& {\bmv}_{N}(\bmx,t), \qquad \bmx \in \Gamma_{\bmu,N}, \quad t > 0, \\ %\nonumber
p(\bmx,0) &=& p_{0}(\bmx), \qquad \bmx \in \Omega, \\ %\nonumber
\bmu (\bmx,0) &=& \bmu_0(\bmx), \qquad \bmx \in \Omega,
\end{eqnarray*}
%\end{subequations}
where 
$\Gamma_{p,D} \cap \Gamma_{p,N} = \emptyset$,
$\overline{\Gamma}_{p,D}\cup \overline{\Gamma}_{p,N}=\Gamma=\partial{\Omega}$,
$\Gamma_{\bmu,D} \cap \Gamma_{\bmu,N} = \emptyset$ and
$\overline{\Gamma}_{\bmu,D} \cup \overline{\Gamma}_{\bmu,N}=\Gamma$
are fulfilled. 

We divide~\eqref{nlBiot1} by $2 \mu$ and~\eqref{nlBiot2} by $\alpha$ and introduce the new variable
$\tilde p = \frac{\alpha}{(2\mu)} p$ and obtain the equivalent system
\begin{subequations}\label{eq:nlBiot_scaled}
\begin{align}
-\div (\ep(\bmu) + \tilde \lambda \div \bm u I) + \nabla \tilde p &= \tilde \bmf~~ \text{in}~~ \Omega\times (0,T),\label{nlBiot_scaled1}\\
\frac{\partial}{\partial t}\left( \div \bmu + \tilde S \tilde p\right) - \div (\tilde K (\div \bmu) \nabla \tilde p)  &= \tilde g\;\;\text{in}~~ \Omega\times (0,T),
\label{nlBiot_scaled2}
\end{align}
\end{subequations}
for the scaled parameters $\tilde \lambda = \lambda/(2 \mu)$, $\tilde S = \frac{2 \mu S}{\alpha^2}$, scaled right-hand sides
$\tilde \bmf = \bmf/(2 \mu)$, $\tilde g = g/\alpha$, and the scaled function $\tilde K  (\div \bmu) = \frac{2 \mu}{\alpha^2} K (\div \bmu)$,
the latter satisfying Assumption~\ref{ass:K} if and only if $K$ does.

Next, on a given subdivision $0=t_0 < t_1 < \ldots < t_{n-1} < t_n < \ldots < t_{N-1} < t_N = T$ of the time interval $[0,T]$ into subintervals $[t_{n-1},t_n]$
of length $\tau:=\tau_n=t_n-t_{n-1}$,  $n=1,2,\ldots,N$, we discretize in time equation~\eqref{nlBiot_scaled2} by the implicit Euler method. Since in what
follows we consider only a single time-step, for convenience, we denote the quantities of interest at time $t_n$ by $\bmu := \bmu_n = \bmu (\bmx,t_n)$ and
$p := \tilde p_n = p (\bmx,t_n)$, and use the short notation $\bmf := \tilde \bmf (\bmx,t_n)$ and $g=\tau \tilde g(\bmx,t_n) + \div \bmu_{n-1} + \tilde S p_{n-1}$
for the right-hand sides in the implicit Euler time-step equations, where we assume that $\bmu_{n-1}$ and $\tilde p_{n-1}$ are known from initial conditions
or have already been computed. Moreover, from now on, we will also skip the tilde symbol again with the scaled parameters in~\eqref{eq:nlBiot_scaled},
including the scaled  hydraulic conductivity.

The variational form of the time-step equations can then be expressed as:
Find $(\bmu,p) \in \bmU \times P$ such that
\begin{subequations}\label{eq:nlBiot_weak}
\begin{align}
(\ep(\bmu),\ep(\bmw)) + \lambda (\div \bmu, \div \bmw) - (p, \div w) &= (\bmf,\bmw) , \qquad \forall \bmw \in \bmU, \label{nlBiot_weak1}\\
(\div \bmu, q) - \tau  (\div K (\div \bmu) \nabla p, q) + S (p,q) &= (g, q), \qquad \forall q \in P,  \label{nlBiot_weak2}
\end{align}
\end{subequations}
where the choice of the spaces $\bmU$ and $P$ depends on the boundary conditions, e.g,, $\bmU = \bmH_0^1(\Omega)$ and $P=L_0^2(\Omega)$
for $\Gamma_{\bmu,D} =\Gamma_{p,N}=\Gamma$. System~\eqref{eq:nlBiot_weak} will be our object of interest in the remainder of this paper.

\begin{remark}
When the hydraulic conductivity $K$ is constant, the model equations become linear.
They are a coupled system of an elliptic equation for the displacement $\bmu$
and a parabolic equation for $p$. The linear problem has been intensively investigated
including the existence and uniqueness of a solution, and its numerical approximations. 
\end{remark}

\subsection{Fixed-stress splitting method}\label{subsec2.2}

The fixed-stress splitting method can be used to solve iteratively system~\eqref{eq:nlBiot_weak}, i.e. for approximating
its solution $(\bmu,p)=(\bmu_n,p_n)$ at time $t=t_n$ starting with an arbitrary initial guess $(\bmu^0,p^0)$ thereby
using an already known approximation of $(\bmu_{n-1},p_{n-1})$, which can be computed by the same method.
In order to avoid confusion with time-stepping, we will use superscripts for the iteration counters in the studied
splitting scheme.
Algorithm~\ref{alg1} below describes this iterative process that consists of computing alternately an update $p^{i+1}$ from a
stabilized mass-balance equation, assuming a constant volumetric mean stress $\lambda \div \bmu^i - L \lambda \alpha^{-1} p^i$,
where $L$ is a free to be chosen stabilization parameter, i.e., solving equation~\eqref{alg1_p} and an update $\bmu^{i+1}$ arising from solution of
the momentum balance equation~\eqref{alg1_u}. The process is continued until a certain convergence criterion is satisfied. 

\begin{algorithm}
\caption{Fixed-stress splitting method in variational form}\label{alg1}
Execute alternately steps (a) and (b) until the convergence criterion is satisfied. \\
(a) Given $\bmu^i$ and $p^i$ and $K^i:=K(\div \bmu^i)$, solve for $p^{i+1}$ the equation
\begin{equation}\label{alg1_p}
\hspace{-1ex} -\tau (\div K^i \nabla p^{i+1},q) + S (p^{i+1},q) + L (p^{i+1},q) = (g,q) + L (p^i,q) - (\div \bmu^i,q), \ \forall q \in P.
\end{equation}
(b) Given $p^{i+1}$, solve for $\bmu^{i+1}$ the equation
\begin{equation}\label{alg1_u}
(\ep (\bmu^{i+1}),\ep (\bmw)) + \lambda (\div \bmu^{i+1},\div \bmw)  = (\bmf,\bmw) + (p^{i+1},\div \bmw), \ \forall \bmw \in \bmU.
\end{equation}
\end{algorithm}

Note that the above algorithm is essentially the classical fixed stress method in which we use the hydraulic conductivity $K^i$ from
the previous iteration in equation~\eqref{alg1_p}.

\begin{remark}
    In the linear case (constant $K$), the iterative splitting schemes have been extensively studied for this multiphysics problem.
    We refer to \cite{MikelicWheeler2013} for a first convergence analysis, then to \cite{Bothetall2017} for an extension to
    heterogeneous flow and by using energy norms. The question of an optimal stabilization parameter $L$ (i.e. the value of $L$
    which leads to the minimal number of iterations) was discussed in detail in~\cite{storvik2019optimization,hong2020parameter},
    where also the connection between the stability of the discretization and the convergence of the splitting scheme has been made. 
    \end{remark}
    
 \begin{remark}
   The fixed-stress splitting scheme presented above can be seen as a combination between a linearization scheme, the $L$-scheme (see \cite{List2016}) and the fixed-stress splitting scheme for linear problems. This was already previously applied for other classes of nonlinear Biot models, as e.g. in \cite{borregales2018robust, borregales2021iterative} or for unsaturated flow in deformable porous media \cite{both2019anderson}. This type of methods are often accelerated by using the Anderson Acceleration method, see e.g. \cite{both2019anderson}.
   \end{remark}
    
%\subsection{Undrained split method}\label{subsec2.3}

\section{Convergence analysis}\label{sec3}

In this section we present the convergence analysis of Algorithm~\ref{alg1} for the iterative solution of the nonlinear poroelasticity
problem~\eqref{eq:nlBiot_weak}.

\subsection{Auxiliary results}\label{subsec3.1}

In order to analyze the fixed-stress split method, i.e., prove its uniform linear convergence for an arbitrary initial guess $(\bmu^0,p^0)$,
we define the errors associated with the approximations $\bmu^i$ and $p^i$ after $i$ executions of steps (a) and (b) of Algorithm~\ref{alg1},
i.e.,
\begin{subequations}\label{eq:errors}
\begin{align}
\euo &:= \bmu^i - \bmu \in \bmU, \label{error_u}\\
\epo &:= p^i - p \in P,  \label{error_p}
\end{align}
\end{subequations}
where $i=0,1,2,\ldots$, and $(\bmu,p)$ denotes the exact solution of the coupled problem~\eqref{eq:nlBiot_weak}.

Subtracting from equations \eqref{alg1_p} and \eqref{alg1_u}, which are solved in the $(i+1)$-st iteration, the
equations~\eqref{nlBiot_weak2} and~\eqref{nlBiot_weak1} results in the error equations
\begin{subequations}\label{eq:error_equations}
\begin{align}
- \tau  (\div K^i \nabla p^{i+1}-\div K \nabla p, q) + S (\epn,q) & \nonumber \\ 
+ L (\epn,q)&= L (\epo,q) -(\div \euo,q), \ \forall q \in P,  \label{error_equations1}\\
(\ep(\eun),\ep(\bmw)) + \lambda (\div \eun, \div \bmw) &= (\epn,\div \bmw), \ \forall \bmw \in \bmU, \label{error_equations2}
\end{align}
\end{subequations}
where we have used the short notation $K:=K(\div \bmu)$. Now, using the test functions $q=\epn$ and $\bmw:=\eun$
in~\eqref{eq:error_equations} and adding up the resulting equations \eqref{error_equations1} and \eqref{error_equations2},
we obtain an error identity, which serves as the basis for our analysis, that is,
\begin{align}\label{eq:error_identity}
\Vert \ep(\eun) \Vert^2 &+ \lambda \Vert \div \eun \Vert^2 + \tau (K^i \nabla p^{i+1} - K \nabla p, \nabla \epn)  \nonumber \\
&+ S \Vert \epn \Vert^2 + L ((\epn-\epo),\epn) = (\div (\eun-\euo),\epn).
\end{align}
The following lemma provides an estimate that can be derived from~\eqref{eq:error_identity}, bounding quantities associated
with the approximations after $(i+1)$ iterations by the norm of the error~$\epo$.

\begin{lemma}\label{lemma1}
Consider the approximations $\bmu^{i+1}$ and $p^{i+1}$ generated via Algorithm~\ref{alg1} and the corresponding errors $\eun$
and $\epn$ defined according to~\eqref{eq:errors}, where $(\bmu,p)$ is the exact solutions of problem~\eqref{eq:nlBiot_weak}.
Then for a stabilization parameter $L \ge 1/(d^{-1}+\lambda)=:1/(c_K^2+\lambda)$ there holds the estimate
\begin{align}\label{basic_estimate}
\frac{1}{2}\left(\Vert \ep(\eun) \Vert^2 + \lambda \Vert {\rm div} \, \eun \Vert^2 \right)
&+\tau  (K^i \nabla p^{i+1} - K \nabla p, \nabla \epn)  \nonumber \\
&+ S \Vert \epn \Vert^2 + \frac{L}{2} \Vert \epn \Vert^2 \le  \frac{L}{2} \Vert \epo \Vert^2.
\end{align}
\end{lemma}
\begin{proof}
In view of
$
(\epn-\epo,\epn)=\frac{1}{2} \left( \Vert \epn - \epo \Vert^2 + \Vert \epn \Vert^2 - \Vert \epo \Vert^2 \right)
$
we can rewrite \eqref{eq:error_identity} in the form
\begin{align}\label{eq:error_identity_r}
\Vert \ep(\eun) \Vert^2 &+ \lambda \Vert \div \eun \Vert^2 + \tau (K^i \nabla p^{i+1} - K \nabla p, \nabla \epn)  \nonumber \\
&+ S \Vert \epn \Vert^2 + \frac{L}{2} \Vert \epn \Vert^2 + \frac{L}{2} \Vert \epn - \epo \Vert^2 \nonumber \\
&= \frac{L}{2} \Vert \epo \Vert^2 +(\div (\eun - \euo), \epn).
\end{align}
Next, for $\bmw=\eun-\euo$, the error equation~\eqref{error_equations2} gives
\begin{align}\label{eq:error_identity_n}
(\div (\eun - \euo), \epn) = (\ep (\eun),\ep (\eun-\euo)) + \lambda (\div \eun, \div (\eun - \euo)).
\end{align}
Now, inserting \eqref{eq:error_identity_n} in \eqref{eq:error_identity_r}, we obtain
\begin{align}\label{eq:estimate_i}
\Vert \ep(\eun) \Vert^2 &+ \lambda \Vert \div \eun \Vert^2 + \tau (K^i \nabla p^{i+1} - K \nabla p, \nabla \epn)  \nonumber \\
&+ S \Vert \epn \Vert^2 + \frac{L}{2} \Vert \epn \Vert^2 + \frac{L}{2} \Vert \epn - \epo \Vert^2 \nonumber \\
&= \frac{L}{2} \Vert \epo \Vert^2 + (\ep (\eun),\ep (\eun-\euo)) + \lambda (\div \eun, \div (\eun - \euo)) \nonumber \\
& \le \frac{L}{2} \Vert \epo \Vert^2 + \frac{1}{2} \left( \Vert \ep(\eun) \Vert^2 + \lambda \Vert \div \eun \Vert^2 \right) \nonumber \\
&+ \frac{1}{2} \left( \Vert \ep(\eun - \euo) \Vert^2 + \lambda \Vert \div (\eun - \euo) \Vert^2 \right),
\end{align}
and, after collecting terms,
\begin{align}\label{eq:estimate_ii}
\frac{1}{2} \left( \Vert \ep(\eun) \Vert^2 + \lambda \Vert \div \eun \Vert^2 \right) &+ \tau (K^i \nabla p^{i+1} - K \nabla p, \nabla \epn)  \nonumber \\
&+ S \Vert \epn \Vert^2 + \frac{L}{2} \Vert \epn \Vert^2 + \frac{L}{2} \Vert \epn - \epo \Vert^2 \nonumber \\
&\le \frac{L}{2} \Vert \epo \Vert^2 \nonumber \\
&+\frac{1}{2} \left( \Vert \ep(\eun - \euo) \Vert^2 + \lambda \Vert \div (\eun - \euo) \Vert^2 \right).
\end{align}
What remains to be done is to estimate the last term on the right-hand side of~\eqref{eq:estimate_ii}. For this purpose, we subtract the
error equation~\eqref{error_equations2} for the $i$-th error $\euo$ from that for the $(i+1)$-st error $\eun$ and test with $\bmw=\eun-\euo$,
resulting in
\begin{align}\label{eq:estimate_iii}
\Vert \ep(\eun - \euo) \Vert^2 + \lambda \Vert \div (\eun - \euo) \Vert^2 & = (\epn - \epo,\div (\eun - \euo)) \nonumber \\
& \le \Vert \epn - \epo \Vert \Vert \div (\eun - \euo) \Vert .
\end{align}
Moreover, using the inequality $\Vert \ep(\bmw) \Vert \ge c_K \Vert \div \bmw \Vert$, which is valid for all $\bmw \in \bmU$
for $c_K = \frac{1}{\sqrt{d}}$ where $d$ is the space dimension, from~\eqref{eq:estimate_iii} we get
$$
(c_K^2+\lambda) \Vert \div (\eun - \euo) \Vert^2 \le \Vert \epn - \epo \Vert \Vert \div (\eun - \euo) \Vert,
$$
or, equivalently,
\begin{align}\label{eq:estimate_iiii}
\Vert \div (\eun - \euo) \Vert \le \frac{1}{c_K^2+\lambda} \Vert \epn - \epo \Vert .
\end{align}
Now, combining~\eqref{eq:estimate_iii} and~\eqref{eq:estimate_iiii},
we find
\begin{align}\label{eq:estimate_aux}
\Vert \ep(\eun - \euo) \Vert^2 + \lambda \Vert \div (\eun - \euo) \Vert^2 &\le \frac{1}{c_K^2+\lambda} \Vert \epn - \epo \Vert^2 \nonumber \\
&\le L \Vert \epn - \epo \Vert^2,
\end{align}
where the last inequality holds due to the assumption $L \ge 1/(c_K^2+\lambda)$.
Finally, using~\eqref{eq:estimate_aux} in~\eqref{eq:estimate_ii} yields~\eqref{basic_estimate},
the estimate we had to prove.
\end{proof}

Before we prove the main result of this paper, we present another lemma that will turn out to be useful. 

\begin{lemma}\label{lemma2}
For the errors associated with the iterates generated by Algorithm~\ref{alg1} there holds the estimate
\begin{align}\label{infsup_estimate}
(\beta_s^{-2}+\lambda)^{-1} \Vert \epn \Vert^2 \le \Vert \ep(\eun) \Vert^2 + \lambda \Vert {\rm div} \, \eun \Vert^2
\end{align}
where $\beta_s$ is the constant in the Stokes inf-sup condition
\begin{align}\label{eq:inf_sup_s}
\inf_{q \in P}
\sup_{\bmw \in \bmU}
\frac{({\rm div} \, \bmw,q)}{\Vert \bmu \Vert_1\Vert q \Vert} \geq \beta_s  .
\end{align}
\end{lemma}

\begin{proof}
As one can easily see, the Stokes inf-sup condition~\eqref{eq:inf_sup_s} implies that for any $\epn$
there exists $\bmw_p \in \bmU$ such that
\begin{align}\label{infsup_equiv}
\div \bmw_p = \epn \quad \mbox{and} \quad \Vert \ep(\bmw_p) \Vert \le \beta_s^{-1} \Vert \epn \Vert.
\end{align}
Hence, this element $\bmw_p$ satifies the estimate
$$
 \Vert \ep(\bmw_p) \Vert^2 + \lambda \Vert \div \bmw_p \Vert^2 \le (\beta_s^{-2}+\lambda) \Vert \epn \Vert^2.
$$
Setting $\bmw=\bmw_p$ in~\eqref{error_equations2} and using \eqref{infsup_equiv}, we obtain
\begin{eqnarray*}
\Vert \epn \Vert^2 &=& (\ep(\eun),\ep(\bmw_p)) + \lambda (\div \eun, \div \bmw_p) \\
& \le & (\Vert \ep(\eun) \Vert^2 + \lambda \Vert \div \eun \Vert^2)^\frac{1}{2} (\Vert \ep(\bmw_p) \Vert^2 + \lambda \Vert \div \bmw_p \Vert^2)^\frac{1}{2} \\
& \le &  (\Vert \ep(\eun)\Vert^2 + \lambda \Vert \div \eun \Vert^2)^\frac{1}{2} (\beta_s^{-2}+\lambda)^\frac{1}{2} \Vert \epn \Vert
\end{eqnarray*}
which shows the desired result.
\end{proof}

\begin{remark}
    The above method of contraction uses solution variables and their natural energy norms.
    This is in the same spirit as the proofs of contraction as carried out
    in~\cite{both2017robust, storvik2019optimization,hong2020parameter,Hong_parameter-robust_2020}.
    This is in contrast to the convergence analysis as carried out in \cite{MikelicWheeler2012, TAKA, AlmaniKumarWheeler2017}
    where the contraction is proved for composition quantities consisting of both pressure and displacement terms.
    For example, in case of fixed stress split scheme, a volumetric mean stress is defined by
    $\sigma = {\lambda} \nabla \cdot \bmu  - {\alpha} p$  and the contraction is proved on this composite
    quantity \cite{MikelicWheeler2012}.  
\end{remark}

\subsection{Main result}\label{subsec3.2}

Lemma~\ref{lemma1} and Lemma~\ref{lemma2} are the key ingredients to establish the uniform convergence of the
fixed-stress split method under Assumption~\ref{ass:K}.

\begin{theorem}\label{theorem1}
Consider the fixed-stress split iteration according to Algorithm~\ref{alg1} to approximate the
exact solution $(\bmu,p)$ of the nonlinear problem~\eqref{eq:nlBiot_weak} in
which the hydraulic conductivity $K=K({\rm div} \, \bmu)$ is a function of the dilation of the solid
that satisfies the conditions in Assumption~\ref{ass:K}.
Further, assume that $\nabla p \in L^\infty(\Omega)$.

Then the iterates $(\bmu^i,p^i)$ converge linearly to $(\bmu,p)$ for a sufficiently small time
step~$\tau$. In particular, the following estimates holds:
\begin{eqnarray}
\Vert \epn \Vert &\le& \sqrt{
\frac{c_0 + \frac{\tau}{4 K_0}   \frac{c^2}{(c_K^2+\lambda)^2}    }{c_1} 
}\Vert \epo \Vert ,  \label{error_estimate_p} \\
\left( \Vert \ep(\eun)\Vert^2 + \lambda \Vert  {\rm div} \, \eun \Vert^2 \right)^\frac{1}{2} &\le& 
\sqrt{\frac{1}{c_K^2+\lambda}} \Vert \epn \Vert , \label{error_estimate_u}
\end{eqnarray}
where $c_0 := L/2 < c_1 := L/2 + (\beta_s^{-2}+\lambda)^{-1}/2$ and $L$ is the stabilization
parameter in Algorithm~\ref{alg1} and $\beta_s$ the Stokes inf-sup constant.
\end{theorem}
\begin{proof}
The inner product $\tau  (K^i \nabla p^{i+1} - K \nabla p, \nabla \epn)$ in~\eqref{basic_estimate}
can be represented as
\begin{align}\label{splitting}
\tau  (K(\div \bmu^i) \nabla p^{i+1} - K (\div \bmu) \nabla p, \nabla \epn) &=
\tau  ( (K(\div \bmu^i) - K (\div \bmu)) \nabla p, \nabla \epn) \nonumber \\
&+ \tau  (K (\div \bmu^i) \nabla \epn, \nabla \epn)
\end{align}
The two terms in this splitting, due to Assumption~\ref{ass:K}, can be estimated as follows:
\begin{eqnarray}
\tau  ( (K(\div \bmu^i) - K (\div \bmu)) \nabla p, \nabla \epn) &\le& \tau c \Vert \div \euo \Vert \Vert \nabla \epn \Vert \nonumber \\
&\le& \tau c \left( \frac{1}{2 \delta} \Vert \div \euo \Vert^2 + \frac{\delta}{2} \Vert \nabla \epn \Vert^2\right) \label{eq:k_estimate1} \\
\tau  ( (K(\div \bmu^i) \nabla \epn, \nabla \epn) &\ge& \tau K_0 \Vert \nabla \epn \Vert^2 \label{eq:k_estimate2}
\end{eqnarray}
where in the first inequality in~\eqref{eq:k_estimate1} we have used the boundedness
of $\nabla p$, i.e., $\Vert \nabla p \Vert_{L^{\infty}} \le c_\infty$ as well as the
Lipschitz continuity of $K$, i.e.,
$$
\vert K(\div \bmu^i) - K (\div \bmu) \vert \le K_L \vert \div \bmu^i - \div \bmu \vert = K_L \vert \div \euo \vert,
$$
which means that we can choose $c := c_\infty K_L$ and the positive constant $\delta$ in the second inequality
in~\eqref{eq:k_estimate1} is still at our disposal. To obtain~\eqref{eq:k_estimate2} we have used the strict positivity
of $K$, see~\eqref{ass1}.

Combining~\eqref{basic_estimate} and~\eqref{infsup_estimate}, we obtain
\begin{align*}%\label{est0}
\frac{1}{2}(\beta_s^{-2}+\lambda)^{-1} \Vert \epn \Vert^2 + \tau  (K^i \nabla p^{i+1} - K \nabla p, \nabla \epn)
+ S \Vert \epn \Vert^2 + \frac{L}{2} \Vert \epn \Vert^2 \le  \frac{L}{2} \Vert \epo \Vert^2
\end{align*}
which implies
\begin{align}\label{est1}
\frac{1}{2}\left( \frac{1}{\beta_s^{-2}+\lambda}+L\right) \Vert \epn \Vert^2 + \tau  (K^i \nabla p^{i+1} - K \nabla p, \nabla \epn)
\le  \frac{L}{2} \Vert \epo \Vert^2.
\end{align}
We rewrite the latter estimate as
\begin{align}\label{est2}
c_1 \Vert \epn \Vert^2 + \tau  (K^i \nabla p^{i+1} - K \nabla p, \nabla \epn)
\le  c_0 \Vert \epo \Vert^2.
\end{align}
using the definitions $0 < c_0 := L/2 < c_1 := L/2 + (\beta_s^{-2}+\lambda)^{-1}/2$.
Taking advantage of the splitting~\eqref{splitting} and the estimates~\eqref{eq:k_estimate1} and~\eqref{eq:k_estimate2},
from~\eqref{est2} we find
\begin{align}\label{est3}
c_1 \Vert \epn \Vert^2 + \tau  K_0 \Vert \nabla \epn\Vert^2 &\le c_0  \Vert \epo \Vert^2 + \tau c \Vert \div \euo \Vert \Vert \nabla \epn \Vert \nonumber \\
&\le c_0  \Vert \epo \Vert^2 + \tau c \frac{1}{c_K^2+\lambda} \Vert \epo \Vert \Vert \nabla \epn \Vert \nonumber \\
&\le c_0  \Vert \epo \Vert^2 + \tau c \frac{1}{c_K^2+\lambda} \left( \frac{1}{2 \delta} \Vert \epo \Vert^2 + \frac{\delta}{2} \Vert \nabla \epn \Vert^2 \right) .
\end{align}
where we have also used the error equation~\eqref{error_equations2} in a similar way as in the derivation of~\eqref{eq:estimate_iiii}.

Now, we choose $\delta := \frac{2 K_0 (c_K^2 + \lambda)}{c}$, to cancel the terms with $\Vert \nabla \epn \Vert^2$ on boths sides
of~\eqref{est3} in order to obtain
\begin{align}\label{est4}
c_1 \Vert \epn \Vert^2 &\le \left( c_0  +\frac{\tau}{4 K_0} \frac{c^2}{(c_K^2+\lambda)^2} \right) \Vert \epo \Vert^2,
\end{align}
the latter being equivalent to~\eqref{error_estimate_p}.

Finally, estimate~\eqref{error_estimate_u} follows in a similar way as~\eqref{eq:estimate_aux}
(from the error equation~\eqref{error_equations2} by setting $\bmw=\eun$ and using
$\Vert \ep(\bmw) \Vert \ge c_K \Vert \div \bmw \Vert$).
\end{proof}

\begin{remark}\label{remark_L}
A simple choice for the stabilization parameter is $L:=1/(c_K^2+\lambda)=1/(d^{-1}+\lambda)$, which together with
the definitions in the Theorem~\ref{theorem1}, i.e., $c_0 := L/2$ and  $c_1 := L/2 + (\beta_s^{-2}+\lambda)^{-1}/2$,
results in a quotient $c_0/c_1=1/(1+\frac{1+2 \lambda}{2 \beta_s^{-2} +2 \lambda})$.
This means that for small values of $\tau$ the theoretical bound on the $L^2$-error contraction factor approaches
$(1+\frac{1+2 \lambda}{2 \beta_s^{-2} +2 \lambda})^{-\frac{1}{2}}$ and for $\lambda \to \infty$ its limit is $1/\sqrt{2}$,
which is in accordance with the corresponding estimate for the linear problem, as presented in~\cite{Hong_fixed-stress_2020}.
\end{remark}

\begin{remark} The number of iterations is obviously depending on the choice of the parameter $L$. To find an optimal $L$ is typically not an easy task. We refer to the work \cite{storvik2019optimization} for a detailed study on how to chose an optimal $L$ for the splitting scheme for the quasi-static linear Biot model.
    \end{remark}
   
%\begin{remark}
%    \textcolor{red}{On the extension of the proof for the case of undrained splitting. What changes? Can this proof be extended without major consequences for the undrained splitting case?}
%\end{remark}
%\subsection{Undrained split method}\label{subsec3.3}

\section{Numerical results}\label{sec4}

The aim of this section is to numerically test the performance of the 
proposed fixed-stress Algorithm~\ref{alg1}.
To this end, we consider the following test problems of type~\eqref{eq:nlBiot} 
which differ only in the definition of the permeability 
coefficient function $K$, namely
\begin{enumerate}[label=(\roman*)]
\item[(o)] $K=K_0$, i.e., a linear model;
\item $K = K(\div \bmu) = K_0 + K_1 (\div \bmu)^2 $;
\item $K= K(\div \bmu) = (K_0 + K_1 \div \bmu)^2$;
\item $K= K(\div \bmu)=K_0 e^{K_1\div \bmu}$,
\end{enumerate}
where the constants $K_0$ and $K_1$ are specified within the description of the individual test settings.

In all computations we have used homogeneous Dirichlet boundary conditions 
for $\bmu$ and homogeneous Neumann boundary conditions for $p$. 
The chosen computational domain $\Omega$ is the L-shaped 
region $\Omega = (0,1)\times (0,1)\setminus [0.5,1)\times [0.5,1)$, 
which is a standard choice in academic 
test setting.

The right-hand sides in all tests are specified as 
$$\bmf = -\div( \ep(\bmu_{\text{ex}}) + \lambda  \div \bmu_{\text{ex}} I)+ \nabla p_{\text{ex}} 
$$
$$
g = -\div \bmu_{\text{ex}}+\tau K_0 \Delta p_{\text{ex}}
- S p_{\text{ex}}
$$
where 
$$
p_{\text{ex}}= \phi(x,y) - \frac{1}{\vert\Omega\vert}\int_{\Omega}\phi(x,y) \text{d}x \text{d}y
\in L_2^0(\Omega), \quad \bmu_{\text{ex}} = 
0.01
\left(\frac{\partial \phi}{\partial y},\frac{\partial \phi}{\partial x}\right)
$$
and
$$
\phi = (\sin(2\pi x) \sin(2 \pi y))^2. 
$$

Following the theoretical findings, the employed stabilization parameter 
$L$ in the fixed-stress iteration is 
\begin{equation}\label{L_numerics}
    L: = L^*= \frac{1}{\lambda+1/2},
\end{equation}
see also Remark~\ref{remark_L}, unless specified explicitly otherwise. 

The numerical tests are performed on meshes with mesh parameter $h$ where all meshes are derived 
via uniform refinement of the coarse mesh given on the left of Fig.~\ref{fig:mesh_and_u}
for which $h=1/16$.

On the right of Fig.~\ref{fig:mesh_and_u} 
we have plotted the solid displacement $\bmu$ which has been computed for the 
linear model for the choice of parameters 
$h=1/32$, $\tau=0.01$,
$S=10^{-4}$, $K_0=10^{-6}$, 
     and $\lambda=10^2 $. The corresponding 
$p$ has been shown on the left of Fig.~\ref{fig:p}, whereas on the right 
of the 
same figure we have presented the solution 
of the non-linear problem (i) for the same choice of parameters where additionally we 
have set $K_1=10^{-1}$.

\begin{figure}[h]
    \centering
    \hspace{-0.3cm}
    \includegraphics[width=0.45\textwidth]{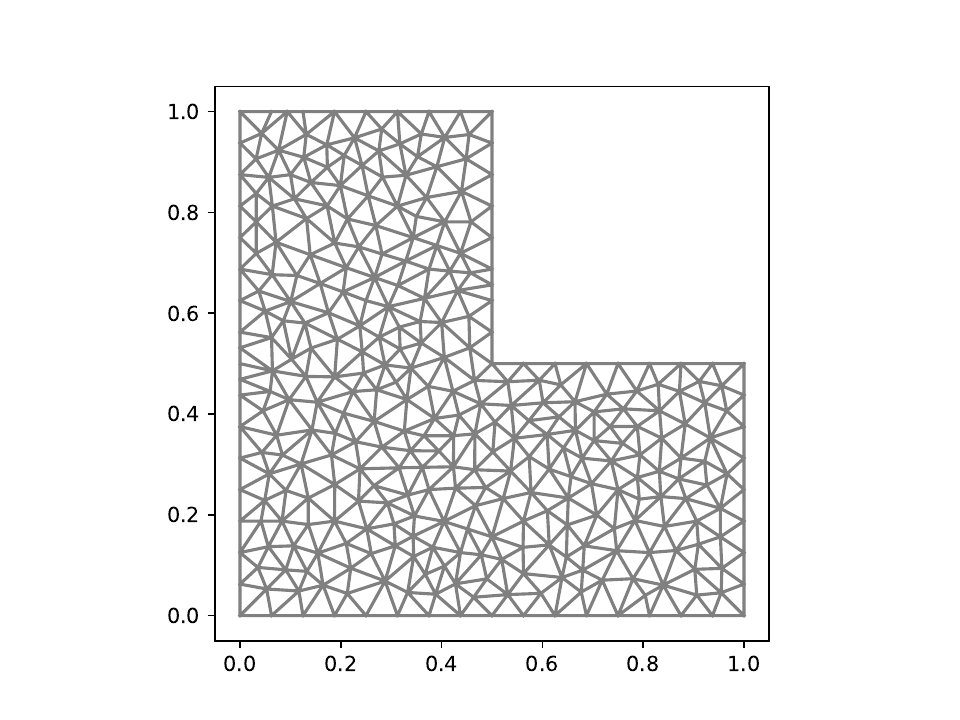}
    \hspace{0.2cm}
    \centering
         \includegraphics[width=0.45\textwidth]{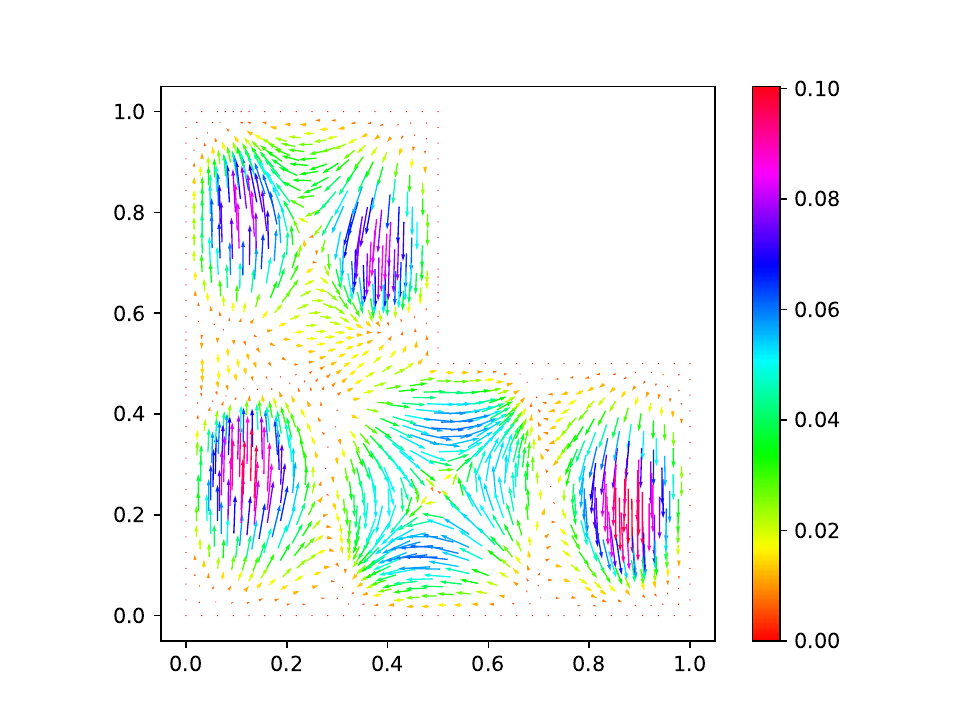}\captionsetup{justification=centering,margin=2cm}
    \caption{Left - triangulation of $\Omega$, $h=1/16$, 
    right - $\bmu$ for model (o), where 
    $h=1/32$, $\tau=0.01$, $S=10^{-4}$, $K_0=10^{-6}$ 
     and $\lambda=10^2 $.}
    \label{fig:mesh_and_u}
\end{figure}

\begin{figure}[h]
     \centering
         \centering
         \includegraphics[width=0.45\textwidth]{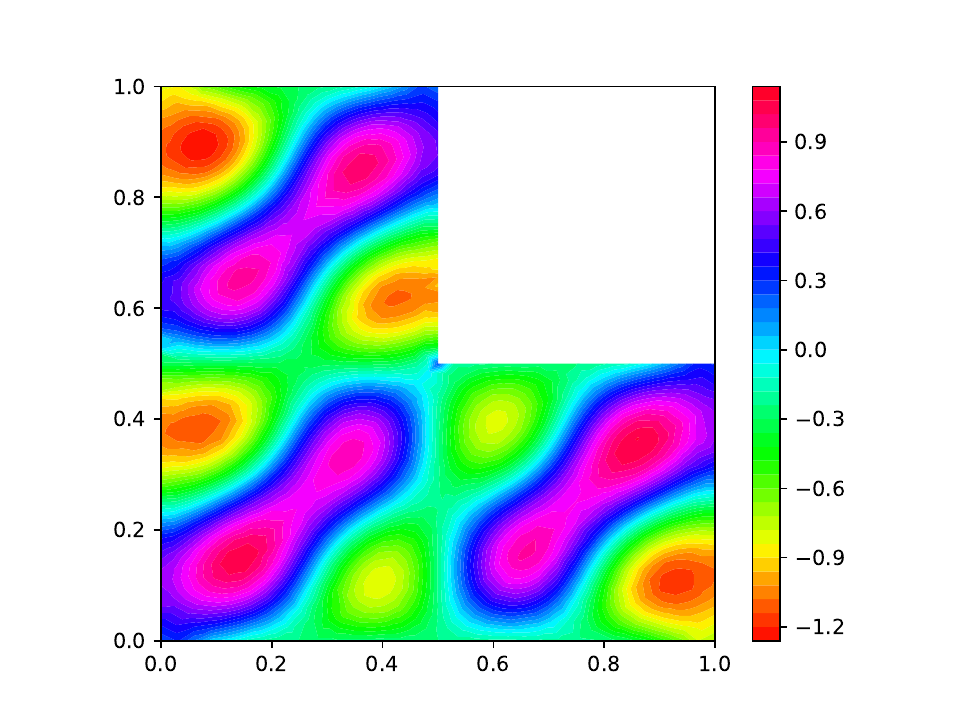}
         \centering
         \includegraphics[width=0.45\textwidth]{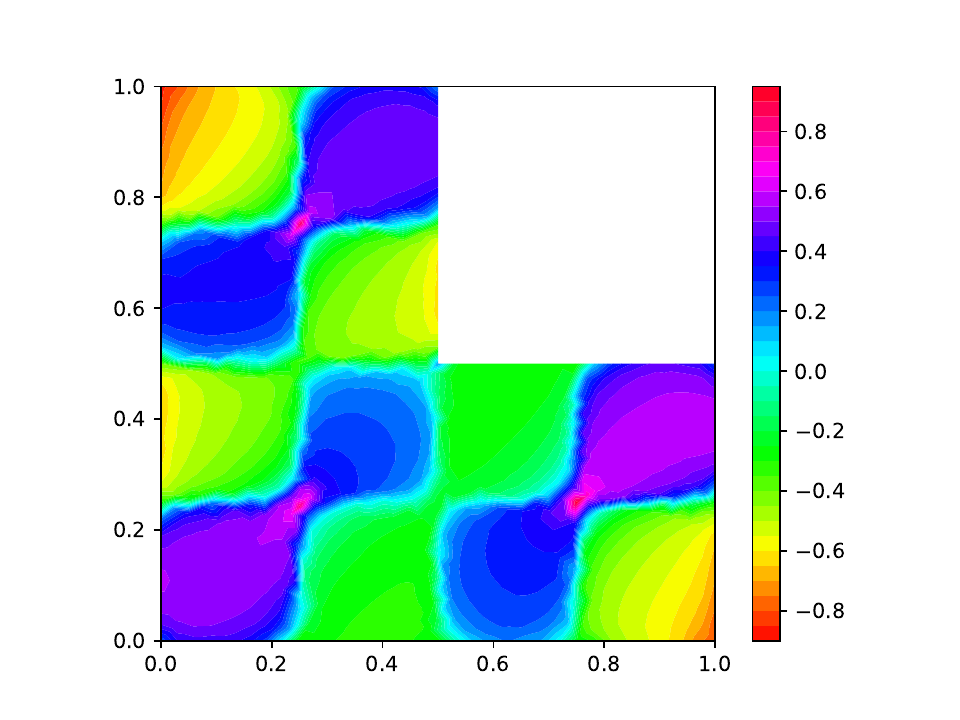}
         \captionsetup{justification=centering,margin=2cm}
         \caption{Left - $p$ for model (o), right - $p$ for model (i). For both $h=1/32$, $\tau=0.01$, $S=10^{-4}$, $K_0=10^{-6}$ 
     and $\lambda=10^2 $, additionally for the non-linear model we 
     have set $K_1=10^{-1}$.}
    \label{fig:p}
\end{figure}

To ascertain the robustness of the proposed fixed-stress method 
we vary the involved model, discretization and stabilization 
parameters. In all test settings we have tried to capture the 
most interesting ranges of the latter and to 
challenge the theoretical results.

The stopping criterion for the fixed-stress algorithm has been chosen to be
\begin{equation}\label{stopping}
    \Vert x_{k+1}-x_{k}\Vert < 10^{-6},
\end{equation}
where $x_{k}$ denotes the $k$-th iterate of 
the proposed fixed-stress method.

The numerical experiments have been performed using the FEniCS 
open-source computing platform for 
solving partial differential equations, \cite{AlnaesBlechta2015a,LoggMardalEtAl2012a}. 

\subsection{Test setting (o)}

In the focus of the first group of experiments is the 
linear model. The main purpose of these tests is to determine how the proposed 
fixed-stress method behaves in terms of number of 
fixed stress iterations when varying the model parameters and 
later compare with the results obtained for the non-linear models.

\begin{figure}[h]
    \centering
    \includegraphics[width=6cm]{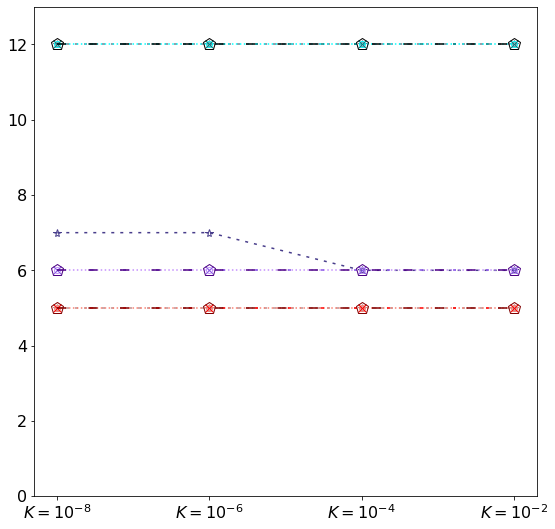}
    \caption{Number of fixed-stress iterations for the linear model where 
    we have set $\tau = 0.01$, $S=10^{-4}$ and $h\in\{1/16,1/32,1/64,1/128\}$. 
    Blue range of colors - $\lambda=10^1$, violet range of colors - $\lambda=10^2$ 
    (single dotted line for $h=1/16$), 
    red range of colors - $\lambda = 10^3$.}
    \label{fig:lin_o}
\end{figure}

As can be seen from Fig.~\ref{fig:lin_o}, the method shows robust behaviour in all tested parameter regimes.

\subsection{Test setting (i)}

Subject of testing in this subsection is the Biot model with 
non-linear permeability coefficient $K(\div \bmu) = K_0 + K_1 (\div \bmu)^2$. 

On Fig.~\ref{fig:nonlin_i_1} the number of fixed-stress iterations are plotted for 
different model and discretization parameters where two different stopping criteria 
have been considered. 
The first is as defined in~\eqref{stopping} and the second one is a residual reduction 
by factor $10^{6}$, i.e.,
\begin{equation*}
    \frac{\Vert r_k\Vert}
    {\Vert r_0\Vert}<10^{-6},
\end{equation*}
where $r_k$ denotes the $k$-th residual of the fixed-stress scheme. We observe that 
the number of iterations when using the second stopping criterion is always smaller and, 
furthermore, that this number is independent of the considered discretization and model 
parameters. Though slightly varying, the number of iterations when using the 
stopping criterion~\eqref{stopping} also show the robustness of the studied iterative method.
\begin{figure}[h]
    \centering
    \includegraphics[width=\textwidth]{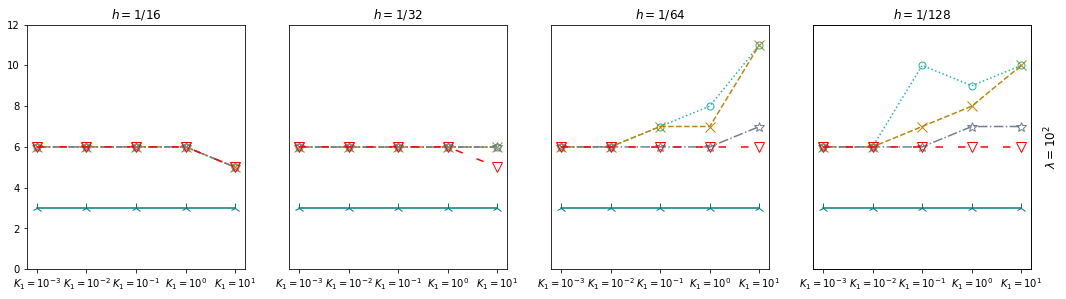}
   \caption{Number of fixed-stress iterations for $K= K_0 + K_1 (\div \bmu)^2$, 
$S=10^{-4}$ where we have set $\tau = 0.01$, $S=10^{-4}$ 
and $\lambda=10^2$. 
Red line - $K_0=10^{-8}$, gray line - $K_0=10^{-6}$, 
yellow line - $K_0 = 10^{-4}$, blue line - $K_0=10^{-2}$. 
The blue-green line shows the results for a different stopping 
criterion, namely a residual reduction by factor $10^{6}$, 
which are identical for the four different choices of $K_0$.
}
    \label{fig:nonlin_i_1}
\end{figure}

For this choice of permeability we have additionally performed a series of tests 
where we have primarily varied $\tau$. It is interesting to observe 
on Fig.~\ref{fig:nonlin_i_2} that though for 
bigger $\tau$ we have the smallest number of iterations when solving on the coarser 
meshes, on the finest mesh the algorithm converges only for the two smallest choices of $\tau$.  
This artefact, however, is in accordance with the analysis presented in this 
paper suggesting 
that for smaller mesh parameter $h$ we might need a smaller time-step $\tau$ to 
guarantee the convergence of the fixed-stress split method.
\begin{figure}[h]
    \centering
    \includegraphics[width=6cm]{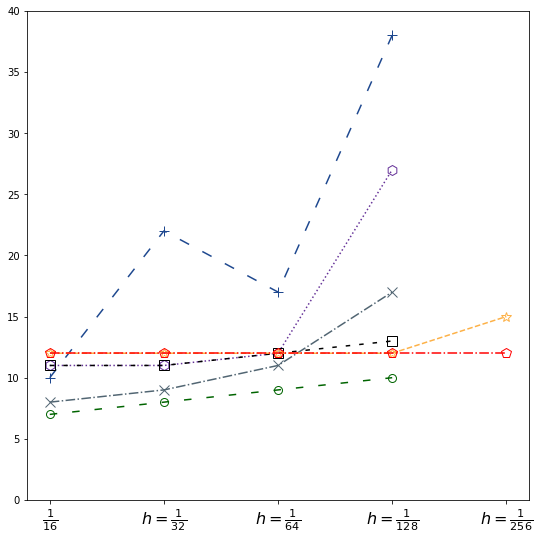}
    \caption{Number of fixed-stress iterations for $K= K_0 + K_1 (\div \bmu)^2$. Red line - 
    $\tau=0.0001$, orange line - $\tau = 0.0005$, 
    black line - $\tau=0.001$, violet line - $\tau=0.005$, blue line - $\tau=0.01$, 
    grey line - $\tau=0.05$, green line - $\tau=0.1$. 
    Missing points means that the algorithm does not converge for the corresponding parameter values.}
    \label{fig:nonlin_i_2}
\end{figure}

%%%%%%%%%

\subsection{Test setting (ii)}

In this subsection we want to test the effect of the stabilization parameter $L$ on the convergence 
of the considered iterative scheme. The studied model here involves the non-linear permeability 
$K(\div \bmu) = (K_0 + K_1 \div \bmu)^2$.
\begin{figure}[h]
    \centering
    \includegraphics[width=\textwidth]{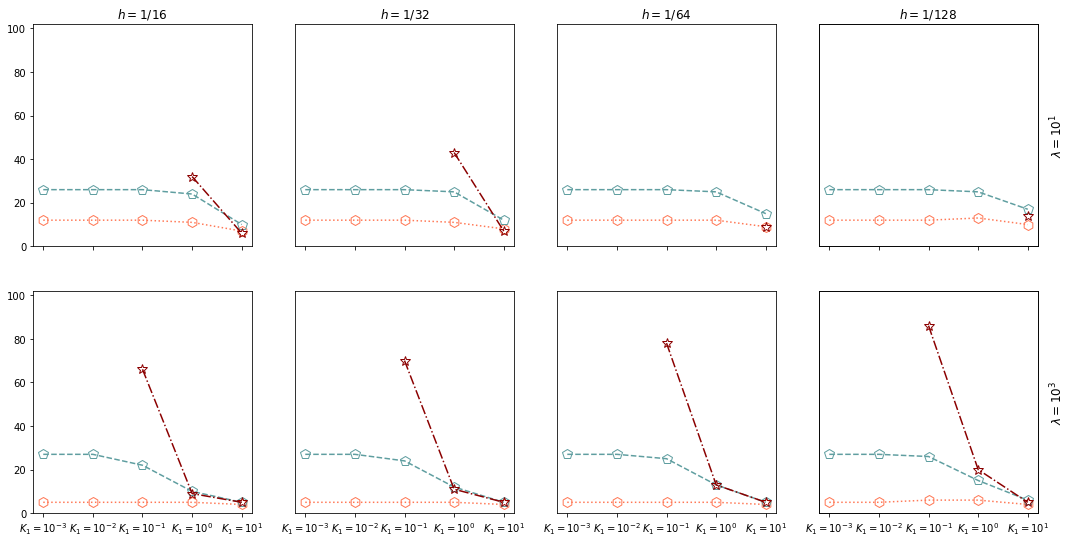}
    \caption{Number of fixed-stress iterations for 
    $K= K(\div \bmu) = (K_0 + K_1 \div \bmu)^2$. 
    We have set $\tau=0.01$ and $S=10^{-4}$. Pink line - stabilization parameter $L=L^*$ as in~\eqref{L_numerics}, 
    blue line $L=2L^*$, red line 
    $L=L^*/2$. Missing points means no convergence within 100 iterations.}
    \label{fig:nonlin_ii}
\end{figure}

The results in Fig.~\ref{fig:nonlin_ii} 
demonstrate the expected robust 
convergent behavior of the method 
for $L=L^*$, chosen as in~\eqref{L_numerics}, as well as for $L=2L^*$, 
%for both of 
which is in accordance 
%we expect convergence according 
with the developed theory.
%, demonstrate that when 
%On the other hand when 
Moreover, when $L$ is too small, as in the example when 
$L=L^*/2$, 
we observe that the algorithm indeed does not converge in some parameter regimes. 
The tests indicate also that the method performs better when 
$L$ is smaller but satisfies the theoretical bounds.

\subsection{Test setting (iii)}

In our last setting we tested the algorithm for a wide range of model and discretization 
parameters when 
$K= K_0 e^{K_1\div \bmu}$. As can be seen from Fig.~\ref{fig:nonlin_iii}, 
the method shows robust behavior with respect to varying all parameters.
\begin{figure}[h]
    \centering
    \includegraphics[width=6cm]{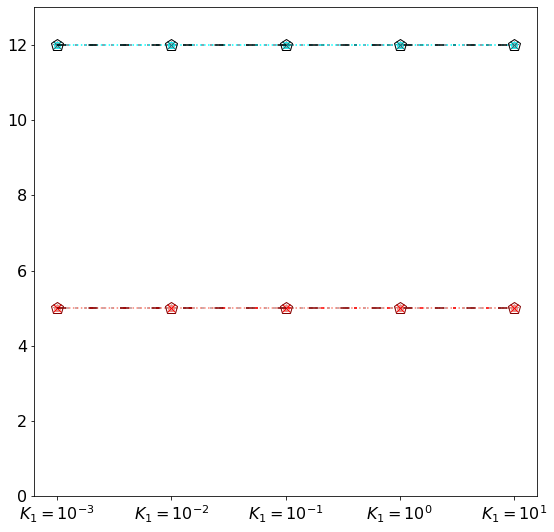}
   \caption{Number of fixed-stress iterations for $K= K_0 e^{K_1\div \bmu}$, 
$S=10^{-4}$, $K_0\in\{10^{-8},10^{-6},10^{-4},10^{-2}\}$, 
$h\in \{1/16,1/32,1/64,1/128\}$. 
Blue range of colors - $\lambda = 10^1$, 
red range of colors - $\lambda= 10^3$.
}
    \label{fig:nonlin_iii}
\end{figure}

\section{Conclusions}\label{sec5}

In this paper we have developed a new 
fixed-stress split method
for a nonlinear poroelasticity model 
in which the
permeability depends on the divergence of the displacement. 
We have proven that for sufficiently small 
time steps this algorithm converges linearly and the error contraction factor 
is strictly less than one, independently of 
all model parameters.

Moreover, our theoretical findings are fully supported and justified by a 
series of numerical tests which also 
have shown the robustness and computational 
efficiency of the proposed fixed-stress method.

\section*{Acknowledgements}
JK acknowledges the funding by the German Science Fund (DFG) - project “Physics-oriented solvers
for multicompartmental poromechanics” under grant number 456235063.
FAR acknowledges the support of the VISTA program, The Norwegian Academy of Science and Letters and Equinor.

\bibliography{FS_nonlin_Biot}

\end{document}